\documentclass{article}
\usepackage{amssymb,amsmath,amsthm,mathrsfs,tikz-cd}
\usepackage[title]{appendix}

\title{On the Rational Cuspidal Divisor Class Groups of Drinfeld Modular Curves $X_0(\mathfrak{p}^r)$}
\author{Sheng-Yang Kevin Ho}
\date{}

\newcommand{\Gal}{\operatorname{Gal}}

\newcommand{\GL}{\operatorname{GL}}
\newcommand{\PGL}{\operatorname{PGL}}
\newcommand{\Div}{\operatorname{Div}}
\newcommand{\divisor}{\operatorname{div}}
\newcommand{\cusp}{\operatorname{cusp}}
\newcommand{\ord}{\operatorname{ord}}
\newcommand{\lcm}{\operatorname{lcm}}

\newcommand{\denominator}{\operatorname{denominator}}
\newcommand{\CC}{\mathcal{C}}
\newcommand{\TT}{\mathcal{T}}
\newcommand{\Tree}{\mathscr{T}}
\newcommand{\p}{\mathfrak{p}}
\newcommand{\n}{\mathfrak{n}}
\newcommand{\m}{\mathfrak{m}}
\newcommand{\dd}{\mathfrak{d}}
\newcommand{\HarZ}{\mathcal{H}(\Tree,\mathbb{Z})}
\newcommand{\HarQ}{\mathcal{H}(\Tree,\mathbb{Q})}

\theoremstyle{plain}
\newtheorem{thm}{Theorem}[section]
\newtheorem{lem}[thm]{Lemma}
\newtheorem{conj}[thm]{Conjecture}

\newtheorem{cor}[thm]{Corollary}
\newtheorem*{mytheorem}{Main Theorem}

\theoremstyle{definition}
\newtheorem{defn}[thm]{Definition}

\theoremstyle{remark}

\newtheorem*{Rem}{Remark}

\NewDocumentCommand{\tens}{e{_^}}{%
  \mathbin{\mathop{\otimes}\displaylimits
    \IfValueT{#1}{_{#1}}
    \IfValueT{#2}{^{#2}}
  }%
}

\newcommand\bigzero{\makebox(0,0){\text{\huge0}}}

\newcommand*{\bord}{\multicolumn{1}{c|}{}}
\newcommand\restr[2]{{
  \left.\kern-\nulldelimiterspace 
  #1 
  \vphantom{\big|} 
  \right|_{#2} 
  }}

\begin{document}
\maketitle

\begin{abstract}
Let $\mathcal{C}(\mathfrak{p}^r)$ be the rational cuspidal divisor class group of the Drinfeld modular curve $X_0(\mathfrak{p}^r)$ for a prime power level $\mathfrak{p}^r\in \mathbb{F}_q[T]$. We relate the rational cuspidal divisors of degree $0$ on $X_0(\mathfrak{p}^r)$ with $\Delta$-quotients, where $\Delta$ is the Drinfeld discriminant function. As a result, we are able to determine explicitly the structure of $\mathcal{C}(\mathfrak{p}^r)$ for arbitrary prime $\mathfrak{p}\in \mathbb{F}_q[T]$ and $r\geq 2$.
\end{abstract}

\section{Introduction}
\subsection{Notation}
\begin{tabular}{p{0.08\textwidth}p{0.8\textwidth}}
$\mathbb{F}_q$ & = finite field of characteristic $p$ with $q$ elements\\
$A$ & = $\mathbb{F}_q[T]$ polynomial ring in $T$ over $\mathbb{F}_q$\\
$K$ & = $\mathbb{F}_q(T)$ rational function field\\
$K_\infty$ & = $\mathbb{F}_q((\pi))$ the completion of $K$ at the infinite place ($\pi:=T^{-1}$)\\
$|\cdot|$ & = $|\cdot|_\infty$ = normalized absolute value on $K_\infty$ ($|T|_\infty:=q$)\\
$\mathcal{O}_\infty$ & = $\mathbb{F}_q[[\pi]]$ ring of integers in $K_\infty$\\
$\mathbb{C}_\infty$ & = the completion of an algebraic closure of $K_\infty$\\
$G$ & = group scheme $\GL(2)$ over $\mathbb{F}_q$\\
$Z$ & = scalar matrices in $G$\\
$\mathcal{K}$ & = $G(\mathcal{O}_\infty)$\\
$\mathcal{I}$ & = $\left\{
\begin{pmatrix}
a & b\\
c & d
\end{pmatrix}
\in \mathcal{K}\bigm\vert c\equiv 0 \bmod \pi\right\}$ Iwahori subgroup of $\mathcal{K}$\\
$\Tree$ & = Bruhat-Tits tree of $\PGL(2,K_\infty)$\\
$V(\Tree)$ & = $G(K_\infty)/\mathcal{K}\cdot Z(K_\infty)$ vertices of $\Tree$\\
$E(\Tree)$ & = $G(K_\infty)/\mathcal{I}\cdot Z(K_\infty)$ oriented edges of $\Tree$
\end{tabular}

\subsection{Motivation}
For a positive integer $N$, let $J_0(N)$ be the Jacobian variety of the classical modular curve $X_0(N)$ and $\TT(N):=J_0(N)(\mathbb{Q})_{\text{tors}}$ its rational torsion subgroup. By the Mordell-Weil theorem, $\TT(N)$ is a finite abelian group. Let $\CC_N$ be the cuspidal subgroup of $J_0(N)$ and $\CC_N(\mathbb{Q})$ its rational subgroup. Let $\CC(N)$ be the rational cuspidal divisor class group of $X_0(N)$; cf. \cite{YOO_2023}. By a theorem of Manin and Drinfeld, $\CC_N$ is a finite group, so we have $$\CC(N)\subseteq \CC_N(\mathbb{Q})\subseteq \TT(N).$$ In the early 1970s, for any prime $p$, Ogg \cite{Ogg} conjectured that $$\CC(p)=\CC_p(\mathbb{Q})=\TT(p)$$
and computed that $\CC(p)$ is a cyclic group generated by $\overline{[0]-[\infty]}$ of order $\frac{p-1}{(p-1, 12)}$. Later in 1977, Mazur \cite{mazur_modular_1977} proved this conjecture by studying the Eisenstein ideal of the Hecke algebra of level $p$. A generalized Ogg's conjecture states as follows:
\begin{conj}[still open]
For any positive integer $N$, $$\CC(N)= \CC_N(\mathbb{Q})= \TT(N).$$
\end{conj} We focus on the structure of $\CC(N)$. In 1997, Ling \cite{ling_q-rational_1997} computed the structure of $\CC(p^r)$, where $p\geq 3$ is a prime and $r\geq 1$. Recently in 2023, by Yoo, Lorenzini, Takagi, Chua, Rouse, Webb, and etc., the
structure of $\CC(N)$ for arbitrary positive $N$ was completely determined. See \cite{YOO_2023} for more details.

In this paper, we study an analogue of the above results in the function field setting. Now, let $\n\in A$ be monic and $\Gamma_0(\n)$ be the congruence subgroup of $\Gamma = G(A)$ consisting of matrices that are upper triangular modulo $\n$. Let $\Omega = \mathbb{C}_\infty - K_\infty$ be the Drinfeld upper half plane. Let $\Gamma_0(\n)$ act on $\Omega$ by linear fractional transformations. Drinfeld proved in \cite{drinfeld_elliptic_1974} that the quotient $\Gamma_0(\n)\backslash \Omega$ is the space of $\mathbb{C}_\infty$-points of an affine curve $Y_0(\n)$ defined over $K$, which is a moduli space of rank-$2$ Drinfeld modules. The unique smooth projective curve over $K$ containing $Y_0(\n)$ is denoted by $X_0(\n)$, which is called the Drinfeld modular curve of level $\n$. Let $J_0(\n)$ be the Jacobian variety of $X_0(\n)$ and $\TT(\n):=J_0(\n)(K)_{\text{tors}}$ its rational torsion subgroup. By the Lang-N\'eron theorem, $\TT(\n)$ is a finite abelian group. Let $\CC_\n$ be the cuspidal subgroup of $J_0(\n)$, which is a subgroup of $J_0(\n)$ generated by the linear equivalence classes of the differences of cusps. Let $\CC_\n(K)$ be the rational cuspidal subgroup of $J_0(\n)$, which is the group of the rational points on $\CC_\n$. Let $\CC(\n)$ be the rational cuspidal divisor class group of $X_0(\n)$, which is a subgroup of $J_0(\n)$ generated by the linear equivalence classes of the degree $0$ rational cuspidal divisors on $X_0(\n)$. By Gekeler \cite{gekeler_note_2000}, $\CC_\n$ is a finite group, so we have $$\CC(\n)\subseteq \CC_\n(K)\subseteq \TT(\n).$$ As an analogue of generalized Ogg's conjecture, we have the following:
\begin{conj}
For any monic $\n\in A$, $$\CC(\n)= \CC_\n(K)= \TT(\n).$$
\end{conj}
This conjecture is true when $\n=\p$ is a prime in $A$ by P{\'a}l \cite{Pal2005} (a prime in $A$ means a monic irreducible element); it is also true when $\n=T^3$ or $T^2(T-1)$ by Papikian and Wei \cite{papikian_eisenstein_2016}. However, the conjecture
is still open in general.

We study the structure of $\CC(\p^r)$ in $\TT(\p^r)$, where $\p\in A$ is a prime and $r\geq 1$. To simplify the notation, denote $M(\p) := \frac{|\p|^2-1}{q^2-1}$ and $$N(\p):=\begin{cases}
\frac{|\p|-1}{q^2-1},  & \text{if $\deg(\p)$ is even.} \\
\frac{|\p|-1}{q-1},  & \text{otherwise.}
\end{cases}$$ There are some known results:
\begin{thm}[Gekeler \cite{gekeler_1997}]\label{Gekeler}
For a prime $\p\in A$, the group $\CC(\p)$ is cyclic of order $N(\p)$ and generated by $\overline{[0]-[\infty]}$.
\end{thm}

\begin{thm}[Papikian and Wei \cite{papikian_eisenstein_2016}]\label{Papikian and Wei}
The group $C(T^3)$ is cyclic of order $q^2$ and generated by $\overline{[0]-[\infty]}$.
\end{thm}

As a main result, we prove the following:
\begin{mytheorem} [Theorem \ref{Main Theorem}]
Fix a prime $\p\in A$ and $r\geq 2$. Let $C_i$, $D_{r-1}$, and $D_0$ be defined in section \ref{section: Main Theorem}. We have $$\CC(\p^r) = \left(\bigoplus_{1\leq i\leq m}\langle \overline{C_i}\rangle\right)\oplus\left(\bigoplus_{m+1\leq i\leq r-2}\langle \overline{C_i-|\p| C_{i+1}}\rangle\right)\oplus \langle \overline{D_{r-1}}\rangle\oplus \langle  \overline{D_0}\rangle,$$ where $m:=\lfloor \frac{r-1}{2}\rfloor$ and
\begin{enumerate}
    \item $\ord(\overline{C_i})=|\p|^{r-i}M(\p)$ for $1\leq i\leq m$.
    \item $\ord(\overline{C_i-|\p| C_{i+1}})=|\p|^{i}M(\p)$ for $m+1\leq i\leq r-2$.
    \item $\ord(\overline{D_{r-1}})=M(\p)$.
    \item $\ord(\overline{D_0})=N(\p)$.
\end{enumerate}
\end{mytheorem}

\begin{Rem} For a prime $\p\in A$ and $r\geq 1$, we observe the followings:
\begin{enumerate}
    \item The group structure of $\CC(\p^r)$ only depends on $r$ and the degree of $\p$.
    \item If $r = 1$ or $2$, then $\CC(\p^r)$ is $p$-torsion free, where $p:=\text{char}(\mathbb{F}_q)$.
    \item If $r \geq 3$, then $\CC(\p^r)$ is $p$-primary if and only if the degree of $\p$ is $1$.
\end{enumerate}
\end{Rem}

\begin{cor}
For a prime $\p\in A$ and $r\geq 1$, the order of $\CC(\p^r)$ is 
$$|\p|^k\cdot M(\p)^{r-1}\cdot N(\p),$$ where
$$k = \begin{cases}
    \frac{3}{4}r^2-2r+1,  & \text{if $r$ is even.} \\
    \frac{3}{4}r^2-2r+\frac{5}{4},  & \text{otherwise.}
    \end{cases}$$
\end{cor}

\subsection{Idea of the proof}
In the following, we briefly discuss the idea of the proof of Main Theorem.

\begin{lem}[Gekeler \cite{gekeler_1997}]\label{cusp rep}
Let $\n\in A$ be monic. The cusps of $X_0(\n)$ are in bijection with $\Gamma_0(\n)\backslash \mathbb{P}^1(K)$. Moreover, every cusp of $X_0(\n)$ has a representative $\begin{bmatrix} \mathfrak{a} \\ \dd \end{bmatrix}$ in $\Gamma_0(\n)\backslash \mathbb{P}^1(K)$, where $\mathfrak{a},\dd\in A$ are monic, $\dd|\n$, and $\gcd(\mathfrak{a},\n)=1$.
\end{lem}

Fix a prime power $\n=\p^r\in A$. A cusp of $X_0(\n)$ with representative $\begin{bmatrix} \mathfrak{a} \\ \dd \end{bmatrix}$ in the above lemma is called of height $\dd$; cf. \cite[(2.4)]{gekeler_fundamental_1995}. Note that the height $\dd$ is uniquely defined by \cite[lemma 3.1]{papikian_eisenstein_2016}. Let $(P_\dd)$ be the sum of all the cusps of $X_0(\n)$ of height $\dd\mid \n$. The cuspidal divisors $(P_\dd)$ are $K$-rational in the sense that they are invariant under $\Gal(\overline{K}/K)$ by \cite[prop. 6.3]{gekeler_invariants_2001}. Indeed, the cusps of $X_0(\n)$ of the same height form an orbit under $\Gal(\overline{K}/K)$. Note that $$[0]:=\begin{bmatrix}
0 \\ 1
\end{bmatrix}=\begin{bmatrix}
1 \\ 1
\end{bmatrix}=(P_1)\text{ and }[\infty]:=\begin{bmatrix}
1 \\ 0
\end{bmatrix}=\begin{bmatrix}
1 \\ \n
\end{bmatrix}=(P_\n)$$ are two rational cusps of $X_0(\n)$.

A modular function on $X_0(\n)$ is a meromorphic function on $\Omega \cup\ \mathbb{P}^1(K)$ which is invariant under the action of $\Gamma_0(\n)$. A modular unit on $X_0(\n)$ is a modular function on $X_0(\n)$ that does not have zeros or poles on $\Omega$. Let $$\Div_{\cusp}^0(X_0(\n))(K):=\left\{C = \sum_{\substack{\dd|\n\\\text{monic}}}a_\dd\cdot (P_\dd)\bigm\vert \deg(C)=0, a_\dd\in \mathbb{Z}\right\}$$ be the group of the degree $0$ rational cuspidal divisors on $X_0(\n)$, where ``rational'' means $\Gal(\overline{K}/K)$-invariant. Let $\mathcal{U}_\n$ be its subgroup consisting of the divisors of modular units. Let $C_i := (P_{\p^i})-\deg(P_{\p^i})\cdot[\infty]\in \Div_{\cusp}^0(X_0(\n))(K)$, where $0\leq i \leq r-1$. Then $$\CC(\n):=\Div_{\cusp}^0(X_0(\n))(K)/\mathcal{U}_\n$$ is generated by $\{\overline{C_0},\overline{C_1},\cdots,\overline{C_{r-1}}\}$. Let $\mathcal{E}_{\n}$ be the group of modular units on $X_0(\n)$. As an analogue of the classical case, we will construct a map $g: \Div_{\cusp}^0(X_0(\n))(K) \rightarrow \mathcal{E}_{\n}\tens_{\mathbb{Z}}\mathbb{Q}$ in section \ref{Section: the key map g}. By the construction, for $C\in \Div_{\cusp}^0(X_0(\n))(K)$, the order of $\overline{C}$ in $\CC(\n)$ is the smallest number $m>0$ such that $g(m C)\in \mathcal{E}_{\n}$, i.e., $m C$ is the divisor of the modular unit $g(m C)$ on $X_0(\n)$. Here, we identify $f\in \mathcal{E}_{\n}$ with $f\otimes 1\in \mathcal{E}_{\n}\tens_{\mathbb{Z}}\mathbb{Q}$.

Fix $C\in \Div_{\cusp}^0(X_0(\n))(K)$. We want to find the order of $\overline{C}$ in $\CC(\n)$. The approach in \cite{ling_q-rational_1997} and \cite{YOO_2023} in the classical setting does not easily transfer to the function field setting due to technical difficulties in proving Ligozat’s proposition, which completely describes the modular units; cf. \cite[prop. 3.5]{YOO_2023}. More precisely, the classical discriminant function has a $24$-th root $\eta(z)$, which is a holomorphic function on the complex upper half plane with no zeros. However, let $\mathcal{O}(\Omega)^\ast$ be the group of non-vanishing holomorphic rigid-analytic functions on $\Omega$, then the Drinfeld discriminant function $\Delta(z)$ only has a maximal $(q-1)$-th root in $\mathcal{O}(\Omega)^\ast$; cf. \cite{gekeler_1997}. If one tries to find a $(q-1)(q^2-1)$-th root (up to constant multiple) of $\Delta(z)$, there is only a formal product in $t^{\frac{1}{q^2-1}}$ by \cite{Gekeler1984} and \cite{GEKELER_product_expansion}:
$$\widetilde{\eta}(z) := t^{\frac{1}{q^2-1}} \prod_{\substack{0\neq \mathfrak{a}\in A\\\text{monic}}}f_\mathfrak{a}(t),$$ which is not in $\mathcal{O}(\Omega)^\ast$. In \cite{GEKELER_product_expansion}, $\widetilde{\pi}A$ is the Carlitz period, $t := t(z) := \exp_{\widetilde{\pi}A}^{-1}(\widetilde{\pi}z)$, and $f_\mathfrak{a}$'s are specific polynomials over $\mathbb{C}_\infty$ derived from the Carlitz module.

Instead of finding a root of $\Delta(z)$, there is a $k$-th root of $\frac{\Delta(z)}{\Delta(\n z)}$ in $\mathcal{O}(\Omega)^\ast$ with 
$$k = \begin{cases}
(q-1)(q^2-1),  & \text{if $\deg(\n)$ is even.} \\
(q-1)^2,  & \text{otherwise.}
\end{cases}$$
For details, see \cite{gekeler_1997} and section \ref{section: D_n}. Although $k$ is still smaller than $(q-1)(q^2-1)$ when $\deg(\n)$ is odd, by rewriting $g(C)$ with roots of $\Delta$-quotients carefully, we are able to find the optimal upper bound $m_1$ of $\ord(\overline{C})$; cf. section \ref{section: D_n} and lemma \ref{upper bound}. To show that $m_1$ is optimal, we consider the followings:
\begin{defn} (van der Put)
Let $R$ be a commutative ring with unity. An $R$-valued harmonic cochain on $\Tree$ is a function $f: E(\Tree)\rightarrow R$ that satisfies
\begin{enumerate}
    \item $$f(e)+f(\overline{e})=0\text{ for all }e\in E(\Tree).$$
    \item $$\sum_{\substack{e\in E(\Tree)\\t(e)=v}}f(e)=0\text{ for all }v\in V(\Tree).$$
\end{enumerate}
Here, for $e\in E(\Tree)$, $t(e)$ is its terminus and $\overline{e}$ is its inversely oriented edge. Denote $\mathcal{H}(\Tree, R)$ the group of $R$-valued harmonic cochains on $\Tree$.
\end{defn}
\begin{thm}[van der Put \cite{VanderPut1981-1982}]
There is a canonical exact sequence of $G(K_\infty)$-modules $$0\rightarrow\mathbb{C}_\infty^\ast\rightarrow\mathcal{O}(\Omega)^\ast\overset{\widetilde{r}}{\rightarrow} \HarZ \rightarrow 0.$$
\end{thm}
The van der Put map $\widetilde{r}$ extends naturally to the map $$\widetilde{r}:\mathcal{E}_{\n}\tens_{\mathbb{Z}}\mathbb{Q}\hookrightarrow\mathcal{O}(\Omega)^\ast\tens_{\mathbb{Z}}\mathbb{Q}\overset{\widetilde{r}\otimes 1}{\rightarrow}\HarZ\tens_{\mathbb{Z}}\mathbb{Q}\hookrightarrow\HarQ.$$ The smallest positive number $m_2$ with $\widetilde{r}\circ g(m_2 C)\in \HarZ$ is a lower bound for $\ord(\overline{C})$. When $m_1 = m_2$, the bounds are optimal. By a further argument if $m_1 > m_2$, we are still able to prove that $m_1=\ord(\overline{C})$; cf. theorem \ref{subtheorem}.

In the final step, the goal is to write the group $\CC(\n)$ as a product of cyclic groups. Note that $\CC(\n)$ is generated by $\{\overline{C_0}, \overline{C_1},\cdots,\overline{C_{r-1}}\}$. However, there are nontrivial relations between $\overline{C_i}$'s in $\CC(\n)$. By applying lemma \ref{linear independence lemma} successively, we construct a modified generating set $\mathcal{B}$ for $\CC(\n)$ defined in the proof of theorem \ref{Main Theorem} so that $\CC(\n)$ can be expressed as a direct sum of cyclic groups generated by the elements in $\mathcal{B}$. The process of determining the elements in $\mathcal{B}$ and finding suitable edges in $E(\Tree)$ that meet the criteria in lemma \ref{linear independence lemma} constitutes the technical aspect. After multiple attempts, the author eventually succeeded in discovering the desired generating set $\mathcal{B}$.

\section{Preliminaries}
\subsection{Connection with $\Delta$-quotients}\label{Section: the key map g}
Fix a prime power $\n=\p^r\in A$. Let $\Delta(z)$ be the Drinfeld discriminant function defined in \cite{gekeler_1997} with $\Delta_\dd(z):= \Delta(\dd z)$ for $\dd|\n$. Then $\Delta_\dd(z)$ are modular forms on $\Omega$ of weight $q^2-1$ and type $0$ for $\Gamma_0(\n)$ for all $\dd|\n$; cf. \cite{gekeler_1997}. The zero orders of $\Delta_\dd(z)$ at the cusps of $X_0(\n)$ are defined in \cite{gekeler_drinfeld_1986}. Let $\begin{bmatrix}
\mathfrak{a} \\ \p^j
\end{bmatrix}$ be a cusp of $X_0(\n)$ in the form of lemma \ref{cusp rep}. By Gekeler \cite[eq. (3.10) and (3.11)]{gekeler_1997}, we have
\begin{align}
&\ord_{\small{\begin{bmatrix}
\mathfrak{a} \\ \p^j
\end{bmatrix}}}\Delta = \frac{q-1}{\rho(j)}|\p|^{r-\min\{2j,r\}} \label{align_1}\\
&\ord_{\small{\begin{bmatrix}
\mathfrak{a} \\ \p^j
\end{bmatrix}}}\Delta_\n = \frac{q-1}{\rho(j)}|\p|^{r-\min\{2(r-j),r\}},\label{align_2}
\end{align}
where $$\rho(j)=\begin{cases}
1,  & \text{if $0<j<r$.} \\
q-1,  & \text{otherwise.}
\end{cases}$$

Fix $1 \leq i < r$, and consider the degeneracy map $f: X_0(\n)\rightarrow X_0(\p^{i})$ defined in \cite[(2.7)]{papikian_eisenstein_2016}. Then we prove the following lemma:
\begin{lem} [cf. {\cite[fig. 9]{papikian_eisenstein_2016}}] \label{ramification index}
The ramification index of the cusp $\begin{bmatrix}
\mathfrak{a} \\ \p^j
\end{bmatrix} \text{mod}~\Gamma_0(\n)$ of $X_0(\n)$ of height $\p^j$ over the cusp $\begin{bmatrix}
\mathfrak{a} \\ \p^j
\end{bmatrix} \text{mod}~\Gamma_0(\p^i)$ of $X_0(\p^i)$ is $$\begin{cases}
|\p|^{\max\{2j, r\}-\max\{2j, i\}},  & \text{if $0\leq j < i$.} \\
\frac{q-1}{\rho(j)}|\p|^{r-\min\{2j,r\}},  & \text{if $i\leq j \leq r$.}
\end{cases}$$
\end{lem}

\begin{proof}
The zero order of $\Delta$ at the cusp $\begin{bmatrix}
\mathfrak{a} \\ \p^j
\end{bmatrix} \text{mod}~\Gamma_0(\p^i)$ of $X_0(\p^i)$ is 
\begin{equation}
\begin{cases}
\frac{q-1}{\rho(j)}|\p|^{i-\min\{2j,i\}},  & \text{if $0\leq j < i$.} \\
1,  & \text{if $i\leq j \leq r$.}
\end{cases}\label{align_3}
\end{equation}
The result follows by dividing equation (\ref{align_1}) with equation (\ref{align_3}).
\end{proof}

For $\p^i\mid \n$, we know the divisor of $\Delta_{\p^i}$ on $X_0(\p^i)$. By the pullback of $f$ and lemma \ref{ramification index}, we also find the divisor of $\Delta_{\p^i}$ on $X_0(\n)$:
$$
\begin{bmatrix}
\divisor(\Delta)\\
\divisor(\Delta_\p)\\
\divisor(\Delta_{\p^2})\\
\vdots\\
\divisor(\Delta_{\p^r})
\end{bmatrix}
=\Lambda(\n)^\intercal\cdot
\begin{bmatrix}
(P_1)\\
(P_\p)\\
(P_{\p^2})\\
\vdots\\
(P_{\p^r})
\end{bmatrix},
$$
where $(P_{\dd})$ is the sum of all the cusps of $X_0(\n)$ of height $\dd\mid \n$, and 
$$\Lambda(\n)^\intercal=
\begin{bmatrix}
|\p|^r & (q-1)|\p|^{r-2} & \cdots & q-1 & 1\\
|\p|^{r-1} & (q-1)|\p|^{r-1} & \ddots & \vdots & \vdots\\
|\p|^{r-2} & (q-1)|\p|^{r-2} & \ddots & (q-1)|\p|^{r-2} & |\p|^{r-2}\\
\vdots & \vdots & \ddots & (q-1)|\p|^{r-1} & |\p|^{r-1}\\
1 & q-1 & \cdots & (q-1)|\p|^{r-2} & |\p|^r
\end{bmatrix}_{0\leq i, j\leq r}
$$
is a matrix with the $(i,j)$-entries defined by
$$\frac{q-1}{\rho(j)}|\p|^{\max\{j,r-j\}-|i-j|}.$$ One can check that the transpose $\Lambda(\n)$ of $\Lambda(\n)^\intercal$ is invertible over $\mathbb{Q}$ with
\begin{align*}
&\Lambda(\n)^{-1} = \frac{1}{(q-1)(|\p|^{r+1}-|\p|^{r-1})}\times \\
&\begin{bmatrix}
(q-1)|\p| & -|\p| &&&& \\
1-q & |\p|^2+1 &&&&& \\
& -|\p| & \ddots & -|\p|^{m(j)}&&& \\
&&& (|\p|^2+1)|\p|^{m(j)-1} &&& \\
&&& -|\p|^{m(j)} & \ddots & -|\p|& \\
&&&&& |\p|^2+1 & 1-q \\
&&&&& -|\p| & (q-1)|\p|
\end{bmatrix},
\end{align*}
where $m(j):=\min\{j,r-j\}$, and the $(i,j)$-entry of $\Lambda(\n)^{-1}$ is
$$
\frac{1}{(q-1)(|\p|^{r+1}-|\p|^{r-1})}\times \begin{cases}
(|\p|^2+1)|\p|^{m(j)-1},  & \text{if $1\leq i=j\leq r-1$.} \\
-|\p|^{m(j)},  & \text{if $|i-j|=1$ and $j\neq 0, r$.} \\
(q-1)|\p|,  & \text{if $(i,j)=(0,0)$ or $(r,r)$.} \\
1-q,  & \text{if $(i,j)=(1,0)$ or $(r-1,r)$.} \\
0,  & \text{otherwise.}
\end{cases}
$$
Recall that $\mathcal{E}_{\n}$ is the group of modular units on $X_0(\n)$. To simplify the notation, we denote $f\otimes a\in \mathcal{E}_{\n}\tens_{\mathbb{Z}}\mathbb{Q}$ formally by $f^a$. Since $f^b \otimes a = f\otimes ba$ for $b\in \mathbb{Z}$, we identify $(f^b)^a$ with $f^{ba}$. We construct the following group homomorphism:
\[\begin{tikzcd}
g: &\Div_{\cusp}^0(X_0(\n))(K) \arrow[r] &\displaystyle\mathcal{E}_{\n}\tens_{\mathbb{Z}}\mathbb{Q}\\
&C=\displaystyle\sum_{\substack{\dd|\n\\\text{monic}}}a_\dd(P_\dd) \arrow[r,mapsto] & \displaystyle\prod_{\substack{\dd|\n\\\text{monic}}}\Delta_\dd^{r_\dd},
\end{tikzcd}\]
where $r_\dd\in \mathbb{Q}$ are defined by
$$\begin{bmatrix}
r_1\\
r_\p\\
r_{\p^2}\\
\vdots\\
r_{\p^r}
\end{bmatrix}=\Lambda(\n)^{-1}\cdot\begin{bmatrix}
a_1\\
a_\p\\
a_{\p^2}\\
\vdots\\
a_{\p^r}
\end{bmatrix}.$$ By the construction, for $C\in \Div_{\cusp}^0(X_0(\n))(K)$ with $g(C) = \prod_{\substack{\dd|\n\\\text{monic}}}\Delta_\dd^{r_\dd}$, we have $\sum_{\substack{\dd|\n\\\text{monic}}}r_\dd\cdot\divisor(\Delta_\dd) = C$. This implies that $\sum_{\substack{\dd|\n\\\text{monic}}}r_\dd=0$ since $\deg(C) = 0$ and $\deg(\divisor(\Delta))=\deg(\divisor(\Delta_\dd))>0$ on $X_0(\n)$ for all $\dd\mid \n$. Thus, $g$ is well-defined as $\frac{\Delta}{\Delta_\dd}\in \mathcal{E}_{\n}$ for all $\dd\mid \n$; cf. \cite{gekeler_1997}. The images of $g$ are called $\Delta$-quotients.

\begin{lem}
Let $\n=\p^r\in A$ be a prime power. The degree of the rational cuspidal divisor $(P_{\p^i})$ of height $\p^i\mid \n$ on $X_0(\n)$ is
$$
\begin{cases}
\frac{|\p|-1}{q-1}|\p|^{\min\{i, r-i\}-1},  & \text{if $0 < i < r$.}\\
1,  & \text{otherwise.}
\end{cases}
$$
\end{lem}

\begin{proof}
Recall that $(P_1) = [0]$ and $(P_\n) = [\infty]$ are two cusps (of degree $1$). For $0 < i < r$, consider $C_i := (P_{\p^i})-\deg(P_{\p^i})\cdot[\infty]\in \Div_{\cusp}^0(X_0(\n))(K)$ with $g(C_i) = \prod_{\substack{\dd|\n\\\text{monic}}}\Delta_\dd^{r_\dd}$. Then we compute that
\begin{align*}
&(q-1)(|\p|^{r+1}-|\p|^{r-1})\sum_{\substack{\dd|\n\\\text{monic}}}r_\dd\\ &= (|\p|-1)^2|\p|^{\min\{i, r-i\}-1}-(q-1)(|\p|-1)\deg (P_{\p^i})=0.
\end{align*}
\end{proof}

\begin{Rem}
For an alternative proof of the above lemma, one can count the number of the cusps of $X_0(\n)$ of the same height directly by lemma 3.1 in \cite{papikian_eisenstein_2016}.
\end{Rem}

\subsection{Evaluation of harmonic cochains $\widetilde{r}(\Delta_\n)$} \label{Section: Evaluation of harmonic cochains}
Recall that $\Tree$ is the Bruhat-Tits tree of $\PGL(2,K_\infty)$ with the vertices $V(\Tree) = G(K_\infty)/\mathcal{K}\cdot Z(K_\infty)$ and the oriented edges $E(\Tree) = G(K_\infty)/\mathcal{I}\cdot Z(K_\infty)$; cf. \cite{GekelerREVERSAT}. We begin with some results from \cite{gekeler_1997}. The set $$S_X:=\left\{\begin{pmatrix}
\pi^k & u\\
0 & 1
\end{pmatrix}\mid k\in \mathbb{Z}, u\in K_\infty, u \bmod \pi^k \mathcal{O}_\infty\right\}$$ is a set of representatives for $V(\Tree)$. Denote $v(k,u)$ the vertex corresponding to $\begin{pmatrix}
\pi^k & u\\
0 & 1
\end{pmatrix}$, and let $e(k,u)$ be the edge pointing to $\infty$ with origin $v(k,u)$. Then we have the following lemma:

\begin{lem}[Gekeler {\cite[cor. 2.9]{gekeler_1997}}]
$$\widetilde{r}(\Delta)(e(j+1,0))=\begin{cases}
-(q-1)q^{-j},  & \text{if $j\leq 0$.} \\
(q-1)(q^{j+1}-q-1),  & \text{otherwise.}
\end{cases}$$
\end{lem}
With the help of the above lemma, we are able to prove the following:
\begin{lem}\label{evaluations}
Let $0\neq \n\in A$ with $\delta=\deg \n$.
\begin{enumerate}
    \item For $j\in \mathbb{Z}$, $$\widetilde{r}(\Delta_\n)(e(j+1,0))=\begin{cases}
    -(q-1)q^{\delta-j},  & \text{if $j\leq \delta$.} \\
    (q-1)(q^{j-\delta+1}-q-1),  & \text{otherwise.}
    \end{cases}$$
    \item For $j \geq 1$, $$\widetilde{r}(\Delta_\n)(e(j+1,\pi^j))
    =-(q-1)q^{|\delta-j|}.$$
\end{enumerate}
\end{lem}

\begin{proof}
\begin{enumerate}
    \item Observe that $$\begin{pmatrix}
    \n & 0 \\
    0 & 1
    \end{pmatrix}\begin{pmatrix}
    \pi^{j+1} & 0 \\
    0 & 1
    \end{pmatrix}=\begin{pmatrix}
    \pi^{j-\delta+1} & 0 \\
    0 & 1
    \end{pmatrix}$$ in $E(\Tree) = G(K_\infty)/\mathcal{I}\cdot Z(K_\infty)$. Then we have \begin{align*}
    &\widetilde{r}(\Delta_\n)(e(j+1,0))\\
    &=\widetilde{r}(\Delta)(\begin{pmatrix}
    \n & 0 \\
    0 & 1
    \end{pmatrix}\begin{pmatrix}
    \pi^{j+1} & 0 \\
    0 & 1
    \end{pmatrix})\\
    &=\widetilde{r}(\Delta)(\begin{pmatrix}
    \pi^{j-\delta+1} & 0 \\
    0 & 1
    \end{pmatrix}) \\
    &=\widetilde{r}(\Delta)(e(j-\delta+1,0)) \\
    &=\begin{cases}
    -(q-1)q^{\delta-j},  & \text{if $j\leq \delta$.} \\
    (q-1)(q^{j-\delta+1}-q-1),  & \text{otherwise.}
    \end{cases}
    \end{align*}
    \item For any $\phi\in \HarZ$ and $j > 0$, we have $$(q-1)\phi(e(j+1,\pi^j))+\phi(e(j+1,0))=\phi(e(j,0)).$$ It follows that
    \begin{align*}
    &\widetilde{r}(\Delta_\n)(e(j+1,\pi^j))\\
    &=(q-1)^{-1}[\widetilde{r}(\Delta_\n)(e(j,0))-\widetilde{r}(\Delta_\n)(e(j+1,0))] \\
    &=-(q-1)q^{|\delta-j|}.
    \end{align*}
\end{enumerate}
\end{proof}

\subsection{A maximal root $D_\n$ of $\Delta/\Delta_{\n}$} \label{section: D_n}
Fix a monic $\n\in A$ of degree $\delta > 0$. Recall that $\Delta$ is the Drinfeld discriminant function. Let $D_{\n}$ be the function defined in \cite[p. 200]{gekeler_1997}. By \cite[cor. 3.18]{gekeler_1997}, $D_\n$ is a maximal $k$-th root (up to constant multiple) of $\frac{\Delta}{\Delta_\n}$ in $\mathcal{O}(\Omega)^\ast$, where $$k = \begin{cases}
(q-1)(q^2-1),  & \text{if $\delta$ is even.} \\
(q-1)^2,  & \text{otherwise.}
\end{cases}$$

Recall the following lemma:
\begin{lem} [{\cite[cor. 3.21]{gekeler_1997}}]
Let $\chi_{\n}: \Gamma_0(\n)\rightarrow \mathbb{F}_q^\ast$ be the character defined in \cite[thm. 3.20]{gekeler_1997}. The function $D_{\n}$ transforms under $\Gamma_0(\n)$ according to the character $$\omega_{\n}:=
\begin{cases}
\chi_{\n}\cdot \det^{\delta/2},  & \text{if $\delta$ is even.} \\
\chi_{\n}^2\cdot \det^{\delta},  & \text{otherwise.}
\end{cases}$$
\end{lem}

From the above, we are able to prove the following:
\begin{lem}\label{transformation law}
Let $0 \neq \m\in A$ and $\gamma\in \Gamma_0(\n\m)$. We have $$D_\n(\m \gamma z) = \omega_{\n}(\gamma)D_\n(\m z).$$
\end{lem}

\begin{proof}
Let $\gamma=\begin{pmatrix}
a & b\\ c\m & d
\end{pmatrix}\in \Gamma_0(\n\m)$. If $\delta$ is even, we have
\begin{align*}
&D_\n(\m\gamma z) = D_\n(\begin{pmatrix}
\m & 0 \\ 0 & 1
\end{pmatrix}\begin{pmatrix}
a & b \\ c\m & d
\end{pmatrix}z)\\
&=D_\n(\begin{pmatrix}
a & b\m \\ c & d
\end{pmatrix}\begin{pmatrix}
\m & 0 \\ 0 & 1
\end{pmatrix}z)\\
&=\chi_\n(\begin{pmatrix}
a & b\m \\ c & d
\end{pmatrix})\det(\begin{pmatrix}
a & b\m \\ c & d
\end{pmatrix})^{\delta/2} D_\n(\m z)\\
&=\chi_\n(\begin{pmatrix}
a & b\\ c\m & d
\end{pmatrix})\det(\begin{pmatrix}
a & b\\ c\m & d
\end{pmatrix})^{\delta/2} D_\n(\m z)\\
&=\chi_\n(\gamma)\det(\gamma)^{\delta/2} D_\n(\m z).
\end{align*} By a similar argument when $\delta$ is odd, we obtain the result.
\end{proof}

In the next section, to obtain the optimal upper bound for the order of an element $\overline{C}$ in $\CC(\p^r)$, we write $g(C)$ into the following form: $$g(C)=\left(\prod_{\dd|\n}\Delta_\dd^{r_\dd}(z)\right)^\frac{1}{(q-1)(|\p|-1)|\p|^{r-1}}=\text{const.} \left(\prod_{\substack{1\leq i\leq r\\0\leq j\leq r-i}}D_{\p^i}^{a_{ij}}(\p^j z)\right)^\frac{k}{(q-1)(|\p|-1)|\p|^{r-1}},$$ where $r_\dd$, $a_{ij}$, and $k$ are integers. Note that expression of $g(C)$ in terms of $D_{\p^i}(\p^j z)$ is not unique. We need to find one with largest possible $k$.

\section{Main Theorem}
\subsection{The order of $\overline{[0]-[\infty]}$ in $\CC(\p^r)$} \label{Section: the order of [0]-[infty]}
Fix a prime power $\p^r\in A$. In this section, we investigate the order of $\overline{C_0} := \overline{[0]-[\infty]}$ in $\CC(\p^r)$. For $r=1$, the result is in theorem \ref{Gekeler}. For $r=2$, we have
\begin{thm}[Gekeler \cite{gekeler_1997}]\label{Gekeler 2}
Let $\p\in A$ be a prime. The order of $\overline{C_0}$ in $\CC(\p^2)$ is $$\ord(\overline{C_0})=\frac{M(\p)}{\gcd(q-1, 2, \deg(\p))}.$$
\end{thm}
For $r\geq 3$, we propose the following:
\begin{thm} \label{subtheorem}
Let $\p\in A$ be a prime and $r\geq 3$. The order of $\overline{C_0}$ in $\CC(\p^r)$ is $$\ord(\overline{C_0})=|\p|^{r-1}\frac{M(\p)}{\gcd(q-1, 2, \deg(\p))}.$$
\end{thm}

\begin{Rem}
Note that the formula in the above theorem for $r\geq 3$ does not specialize to the formulas in theorem \ref{Gekeler} or \ref{Gekeler 2} for $r=1$ or $2$.
\end{Rem}

The proof of theorem \ref{subtheorem} will be provided after some preliminary discussions. Now, we assume that $r\geq 3$. First, we want to find a lower bound for the order of $\overline{C_0}$ in $\CC(\p^r)$. We have
\begin{align*}
& g(C_0)=\left(\Delta^{|\p|}\Delta_{\p}^{-1}\Delta_{\p^{r-1}}\Delta_{\p^r}^{-|\p|}\right)^{\frac{1}{(|\p|^2-1)|\p|^{r-1}}}\\
&=\left(\left(\frac{\Delta_{}}{\Delta_{\p^r}}\right)^{|\p|}\left(\frac{\Delta_{\p^{r-1}}}{\Delta_{\p}}\right)\right)^{\frac{1}{(|\p|^2-1)|\p|^{r-1}}}.
\end{align*}
By \cite[cor. 3.18]{gekeler_1997}, $\frac{\Delta_{\p^{r-1}}}{\Delta_{\p}}$ has no $p$-th root in $\mathcal{O}(\Omega)^\ast$, so we have $|\p|^{r-1}\mid\ord(\overline{C_0})$. Moreover, by lemma \ref{evaluations}, $$|\p|^{r-1}\cdot\widetilde{r}(g(C_0))(e(2,\pi))=(q-1)\frac{|\p|^{r-1}}{q}-\frac{q-1}{M(\p)}\frac{|\p|}{q}.$$ Then we obtain the following lemma:
\begin{lem}\label{lower bound}
Let $\p\in A$ be a prime and $r\geq 3$. Then $$|\p|^{r-1}\cdot\denominator\left(\frac{q-1}{M(\p)}\right)=|\p|^{r-1}\frac{M(\p)}{\gcd(q-1, \deg(\p))}$$ divides the order of  $\overline{C_0}$ in $\CC(\p^r)$.
\end{lem}

\begin{proof}
Let $m=\denominator\left(\frac{q-1}{M(\p)}\right)$. Then $m$ is the smallest positive number such that $m|\p|^{r-1}\cdot\widetilde{r}(g(C_0))(e(2,\pi))$ is integral.
\end{proof}

Second, we want to find an upper bound for $\ord(\overline{C_0})$.
\begin{lem}\label{upper bound}
Let $\p\in A$ be a prime and $r\geq 3$. The order of $\overline{C_0}$ in $\CC(\p^r)$ divides
$$|\p|^{r-1}\frac{M(\p)}{\gcd(q-1, 2, \deg(\p))}.$$
\end{lem}

\begin{proof}
If $\deg(\p)$ or $r$ is even, write $$g(C_0)=\left(\left(\frac{\Delta_{}}{\Delta_{\p^r}}\right)^{|\p|}\left(\frac{\Delta_{\p^{r-1}}}{\Delta_{\p}}\right)\right)^{1/(|\p|^{r+1}-|\p|^{r-1})}.$$ Otherwise, write $$g(C_0)=\left(\left(\frac{\Delta_{\p^{r-1}}}{\Delta_{}}\right)\left(\frac{\Delta_{\p^r}}{\Delta_{\p}}\right)\left(\frac{\Delta_{}}{\Delta_{\p^r}}\right)^{|\p|+1}\right)^{1/(|\p|^{r+1}-|\p|^{r-1})}.$$
Consider $f\in \mathcal{O}(\Omega)^\ast$ defined by
$$f(z)=
\begin{cases}
D_{\p^r}^{|\p|}(z)D_{\p^{r-2}}^{-1}(\p z),  & \text{if $\deg(\p^r)$ is even.} \\
D_{\p^{r-1}}^{-1}(z)D_{\p^{r-1}}^{-1}(\p z)D_{\p^r}^{(|\p|+1)/(q+1)}(z),  & \text{otherwise.}
\end{cases}$$ Then
$$g(C_0)=\text{const.}~f^\frac{(q-1)(q^2-1)}{|\p|^{r+1}-|\p|^{r-1}}.$$ Using lemma \ref{transformation law}, we have $$f(\gamma z)=\chi_\p^2(\gamma)\det(\gamma)^{\deg(\p)}f(z) \text{ for }\gamma\in \Gamma_0(\p^r).$$ Moreover, let $m=\gcd(q-1, 2, \deg(\p))$, then $$(\chi_\p^2(\gamma)\det(\gamma)^{\deg(\p)})^{(q-1)/m}=1\text{ for }\gamma\in \Gamma_0(\p^r).$$ Since $f^{(q-1)/m}$ is a modular unit on $X_0(\p^r)$ and $$g(C_0)=\text{const.}~(f^{(q-1)/m})^{\frac{(q^2-1)m}{|\p|^{r+1}-|\p|^{r-1}}},$$ the order of $\overline{C_0}$ divides 
\begin{equation*}
\denominator\left(\frac{(q^2-1)m}{|\p|^{r+1}-|\p|^{r-1}}\right) = |\p|^{r-1}\frac{|\p|^2-1}{(q^2-1)\gcd(q-1, 2, \deg(\p))}.
\end{equation*}
\end{proof}

\begin{proof} [Proof of Theorem \ref{subtheorem}]
Let $f\in \mathcal{O}(\Omega)^\ast$ defined in lemma \ref{upper bound}. Consider $$\restr{\chi_\p^2}{\Gamma_0(\p^r)}\cdot{\det}^{\deg(\p)}: \Gamma_0(\p^r) \longrightarrow \mathbb{F}_q^\ast.$$ The order $s$ of this character is the size of its image in $\mathbb{F}_q^\ast$, which is also the smallest number such that $f^s$ is a modular unit on $X_0(\p^r)$. Observe that 
$$\left\{\left(\chi_{\p}(\gamma),\det(\gamma)\right)\mid \gamma\in \Gamma_0(\p^r)\right\} = \mathbb{F}_q^\ast\times \mathbb{F}_q^\ast.$$ Then we have
\begin{align*}
s &= \lcm\left(\ord\left(\restr{\chi_\p^2}{\Gamma_0(\p^r)}\right),\ord\left({\det}^{\deg(\p)}\right)\right)\\
&=\lcm\left(\frac{q-1}{\gcd(q-1,2)},\frac{q-1}{\gcd(q-1,\deg(\p))}\right)\\
&=\frac{q-1}{\gcd(q-1,2,\deg(\p))}.
\end{align*}
By lemma \ref{upper bound}, $$g(C_0)=\text{const.}~f^\frac{(q-1)(q^2-1)}{|\p|^{r+1}-|\p|^{r-1}}=\text{const.}~(f^s)^\frac{(q-1)(q^2-1)}{(|\p|^{r+1}-|\p|^{r-1})s}.$$ Moreover, by lemma \ref{lower bound}, $$\ord(\overline{C_0}) \geq \denominator\left(\frac{(q-1)(q^2-1)}{|\p|^{r+1}-|\p|^{r-1}}\right).$$ Hence, the order of $\overline{C_0}$ in $\CC(\p^r)$ is $$\denominator\left(\frac{(q-1)(q^2-1)}{(|\p|^{r+1}-|\p|^{r-1})s}\right)=|\p|^{r-1}\frac{|\p|^2-1}{(q^2-1)\gcd(q-1, 2, \deg(\p))}.$$
\end{proof}

\subsection{The structure of $\CC(\p^r)$} \label{section: Main Theorem}
Fix a prime $\p\in A$ and $r\geq 2$. In this section, we compute the structure of $\CC(\p^r)$. Recall that $(P_{\p^i})$ is the sum of all the cusps of $X_0(\p^r)$ of height $\p^i\mid \p^r$, and $C_i := (P_{\p^i})-\deg(P_{\p^i})\cdot[\infty]\in \Div_{\cusp}^0(X_0(\p^r))(K)$. Define $D_0$ and $D_{r-1}\in\Div_{\cusp}^0(X_0(\p^r))(K)$ in the followings. Let
$$D_0 := C_0+(q-1)\left(\sum_{1\leq i\leq \lfloor \frac{r}{2}\rfloor}C_i+\sum_{\lfloor \frac{r}{2}\rfloor+1\leq i\leq r-1}|\p|^{2i-r}C_i\right).$$
\begin{enumerate}
    \item If $r = 2$, let $D_{r-1} := C_1$.
    \item If $r\geq 3$ and $r\equiv 3 \mod 4$, let
    \begin{align*}
    D_{r-1} &:= C_{r-1}-(|\p|^r-|\p|^{r-2})C_1\\&+\sum_{2\leq i\leq \frac{r-1}{2}}(|\p|^{r-1} - |\p|^{r-2} - |\p|^{r-2i+1} + |\p|^{r-2i})C_i\\&-\sum_{\frac{r+1}{2}\leq i\leq r-2}(|\p|^i - |\p|^{\frac{r-1}{2}} + |\p|^{i-\frac{r-1}{2}} - 1)(C_i-|\p|C_{i+1}).
    \end{align*}
    \item If $r\geq 4$ and $r\equiv 0 \mod 4$, let
    \begin{align*}
    D_{r-1} &:= C_{r-1}-(|\p|^r-|\p|^{r-2})C_1\\&+\sum_{2\leq i\leq \frac{r}{2}-1}(|\p|^{r-1} - |\p|^{r-2} - |\p|^{r-2i+1} + |\p|^{r-2i})C_i\\&+\sum_{\substack{\frac{r}{2}\leq i\leq r-2\\\text{$i$: even}}}(|\p|^{i+1}-2|\p|^i+|\p|^{\frac{r}{2}}-|\p|^{i-\frac{r}{2}+1}+1)(C_i-|\p|C_{i+1})\\&-\sum_{\substack{\frac{r}{2}+1\leq i\leq r-3\\\text{$i$: odd}}}(|\p|^{i+1}-|\p|^{\frac{r}{2}}+|\p|^{i-\frac{r}{2}+1}-1)(C_i-|\p|C_{i+1}).
    \end{align*}
    \item If $r\geq 5$ and $r\equiv 1 \mod 4$, let
    \begin{align*}
    D_{r-1} &:= C_{r-1}-(|\p|^r-|\p|^{r-2})C_1\\&+\sum_{2\leq i\leq \frac{r-1}{2}}(|\p|^{r-1} - |\p|^{r-2} - |\p|^{r-2i+1} + |\p|^{r-2i})C_i\\&-\sum_{\substack{\frac{r+1}{2}\leq i\leq r-2\\\text{$i$: odd}}}(2|\p|^{i+1} - |\p|^i - |\p|^{\frac{r-1}{2}} + |\p|^{i-\frac{r-1}{2}} - 1)(C_i-|\p|C_{i+1})\\&-\sum_{\substack{\frac{r+3}{2}\leq i\leq r-3\\\text{$i$: even}}}(|\p|^i - |\p|^{\frac{r-1}{2}} + |\p|^{i-\frac{r-1}{2}} - 1)(C_i-|\p|C_{i+1}).
    \end{align*}
    \item If $r\geq 6$ and $r\equiv 2 \mod 4$, let
    \begin{align*}
    D_{r-1} &:= C_{r-1}-(|\p|^r-|\p|^{r-2})C_1\\&+\sum_{2\leq i\leq \frac{r}{2}-1}(|\p|^{r-1} - |\p|^{r-2} - |\p|^{r-2i+1} + |\p|^{r-2i})C_i\\&-\sum_{\frac{r}{2}\leq i\leq r-2}(|\p|^{i+1} - |\p|^{\frac{r}{2}} + |\p|^{i-\frac{r}{2}+1} - 1)(C_i-|\p|C_{i+1}).
    \end{align*}
\end{enumerate}

We state Main Theorem in the following:
\begin{thm} \label{Main Theorem}
Let $\p\in A$ be a prime and $r\geq 2$. Then $$\CC(\p^r) = \left(\bigoplus_{1\leq i\leq m}\langle \overline{C_i}\rangle\right)\oplus\left(\bigoplus_{m+1\leq i\leq r-2}\langle \overline{C_i-|\p| C_{i+1}}\rangle\right)\oplus \langle  \overline{D_{r-1}}\rangle\oplus \langle  \overline{D_0}\rangle,$$ where $m:=\lfloor \frac{r-1}{2}\rfloor$ and
\begin{enumerate}
    \item $\ord(\overline{C_i})=|\p|^{r-i}M(\p)$ for $1\leq i\leq m$.
    \item $\ord(\overline{C_i-|\p| C_{i+1}})=|\p|^{i}M(\p)$ for $m+1\leq i\leq r-2$.
    \item $\ord(\overline{D_{r-1}})=M(\p)$.
    \item $\ord(\overline{D_0})=\begin{cases}
    \frac{|\p|-1}{q^2-1},  & \text{if $\deg(\p)$ is even.} \\
    \frac{|\p|-1}{q-1},  & \text{otherwise.}
    \end{cases}$
\end{enumerate}
\end{thm}

Before we prove Main Theorem, we need some preliminary lemmas:
\begin{lem} \label{lemma of the structure of C(p^r)}
The exponent of the group $\CC(\p^r)$ divides $|\p|^{r-1}M(\p)$.
\end{lem}

\begin{proof}
The group $\CC(\p^r)$ is generated by $\{\overline{C_0},\overline{C_1},\cdots, \overline{C_{r-1}}\}$. By theorem \ref{Gekeler 2} and \ref{subtheorem}, $\ord(\overline{C_0})\mid |\p|^{r-1}M(\p)$. It suffices to show that $\ord(\overline{C_i})\mid |\p|^{r-1}M(\p)$ for $1\leq i \leq r-1$. In the following, we compute
$$g(C_i)=\left(\prod_{\dd|\p^r}\Delta_\dd^{r_\dd}\right)^{\frac{1}{(q-1)(|\p|^2-1)|\p|^{\max\{i,r-i\}}}}.$$
\begin{enumerate}
    \item If $1\leq i \leq r-3$, then
    $$\prod_{\dd|\p^r}\Delta_\dd^{r_\dd}=\Delta_{\p^{i-1}}^{-|\p|}\Delta_{\p^{i}}^{|\p|^2+1}\Delta_{\p^{i+1}}^{-|\p|}\Delta_{\p^{r-1}}^{|\p|-1}\Delta_{\p^{r}}^{-|\p|^2+|\p|}.$$
    If $\deg(\p)$ or $r-i$ is even, write $$\prod_{\dd|\p^r}\Delta_\dd^{r_\dd}=
        \left(\frac{\Delta_{\p^{r-1}}}{\Delta_{\p^{i-1}}}\right)^{|\p|}\left(\frac{\Delta_{\p^{r-1}}}{\Delta_{\p^{i+1}}}\right)^{|\p|}\left(\frac{\Delta_{\p^{i}}}{\Delta_{\p^{r}}}\right)^{|\p|^2+1}\left(\frac{\Delta_{\p^{r}}}{\Delta_{\p^{r-1}}}\right)^{|\p|+1}.$$
    Otherwise, write
        $$\prod_{\dd|\p^r}\Delta_\dd^{r_\dd}=
        \left(\frac{\Delta_{\p^{r}}}{\Delta_{\p^{i-1}}}\right)^{|\p|}\left(\frac{\Delta_{\p^{r}}}{\Delta_{\p^{i+1}}}\right)^{|\p|}\left(\frac{\Delta_{\p^{i}}}{\Delta_{\p^{r-1}}}\right)^{|\p|^2+1}\left(\frac{\Delta_{\p^{r-1}}}{\Delta_{\p^{r}}}\right)^{|\p|^2+|\p|}.$$
    \item If $r\geq 3$ and $i=r-2$, then 
    \begin{align*}
    &\prod_{\dd|\p^r}\Delta_\dd^{r_\dd}=\Delta_{\p^{r-3}}^{-|\p|}\Delta_{\p^{r-2}}^{|\p|^2+1}\Delta_{\p^{r-1}}^{-1}\Delta_{\p^{r}}^{-|\p|^2+|\p|} \\
    &=\left(\frac{\Delta_{\p^{r-1}}}{\Delta_{\p^{r-3}}}\right)^{|\p|}\left(\frac{\Delta_{\p^{r-2}}}{\Delta_{\p^{r}}}\right)^{|\p|^2+1}\left(\frac{\Delta_{\p^{r}}}{\Delta_{\p^{r-1}}}\right)^{|\p|+1}.
    \end{align*}
    \item If $i=r-1$, then 
    $$\prod_{\dd|\p^r}\Delta_\dd^{r_\dd} = \Delta_{\p^{r-2}}^{-|\p|}\Delta_{\p^{r-1}}^{|\p|^2+|\p|}\Delta_{\p^{r}}^{-|\p|^2} = \left(\frac{\Delta_{\p^{r}}}{\Delta_{\p^{r-2}}}\right)^{|\p|}\left(\frac{\Delta_{\p^{r-1}}}{\Delta_{\p^{r}}}\right)^{|\p|^2+|\p|}.$$
\end{enumerate}
In each case, by extracting maximal roots of the $\Delta$-quotients $\frac{\Delta_{\p^a}}{\Delta_{\p^b}}$ in $\mathcal{O}(\Omega)^\ast$, one can find a $(q-1)(q^2-1)$-th root $f$ of $\prod_{\dd|\p^r}\Delta_\dd^{r_\dd}$ in $\mathcal{O}(\Omega)^\ast$, which turns out to be a modular unit on $X_0(\p^r)$, i.e., $f$ is invariant under $\Gamma_0(\p^r)$. This process is similar to computation in the proof of lemma \ref{upper bound}. From the above, we obtain that $\ord(\overline{C_i})$ divides $|\p|^{\max\{i,r-i\}}M(\p)$ for $1\leq i \leq r-1$.
\end{proof}

\begin{lem} \label{linear independence lemma}
Let $\n\in A$ be monic and $D_i\in \Div_{\cusp}^0(X_0(\n))(K)$ for all $1\leq i\leq k$. Suppose that there exist $e_1, \cdots, e_\ell\in E(\Tree)$ such that 
\begin{enumerate}
    \item The order of $\overline{D_1}$ in $\CC(\n)$ is equal to $$\lcm\{\denominator(\widetilde{r}(g(D_1))(e_j)):~1\leq j \leq \ell\}.$$
    \item $\widetilde{r}(g(D_i))(e_j)\in \mathbb{Z}$ for all $2\leq i\leq k$ and $1\leq j \leq \ell$.
\end{enumerate}
Then we have $\langle \overline{D_i}:~1\leq i\leq k\rangle = \langle \overline{D_1}\rangle \oplus \langle \overline{D_i}:~2\leq i\leq k\rangle\subseteq \CC(\n)$.
\end{lem}

\begin{proof}
Assume that there is a relation $a_1 \overline{D_1} + \cdots + a_k \overline{D_k} = \overline{a_1 D_1 + \cdots + a_k D_k} = 0$ in $\CC(\n)$ with $a_i\in\mathbb{Z}$. Then $D := a_1 D_1 + \cdots + a_k D_k\in \mathcal{U}_\n$ and $\widetilde{r}(g(D)) \in \HarZ$. Therefore, by evaluating $\widetilde{r}(g(D)) = a_1\cdot \widetilde{r}(g(D_1))+\cdots + a_k\cdot \widetilde{r}(g(D_k))$ on each $e_1, \cdots, e_\ell$ and the assumptions in the lemma, we see that $\ord(\overline{D_1})$ divides $a_1$, which implies that $a_1 \overline{D_1}=0$ and also $a_2 \overline{D_2}+\cdots + a_k \overline{D_k} = 0$.
\end{proof}

Now, we are able to prove Main Theorem in the following.
\begin{proof}[Proof of Theorem \ref{Main Theorem}]
Fix a prime $\p \in A$ of degree $\delta > 0$. Let $r \geq 2$ with $m=\lfloor \frac{r-1}{2}\rfloor$. Define $D_i := C_i$ for $1\leq i\leq m$ and $D_i := C_i-|\p| C_{i+1}$ for $m+1\leq i\leq r-2$. Note that $D_{r-1}$ and $D_0$ are defined above. Since $\CC(\p^r)$ is generated by $\{\overline{C_0},\overline{C_1},\cdots, \overline{C_{r-1}}\}$, one can check that $\mathcal{B}:=\{\overline{D_i}:~0\leq i \leq r-1\}$ is also a generating set for $\CC(\p^r)$. We claim that $\CC(\p^r)=\bigoplus_{i=0}^{r-1}\langle \overline{D_i}\rangle$. By lemma \ref{lemma of the structure of C(p^r)}, $\ord(\overline{D_i})\mid |\p|^{r-i}M(\p)$ for $1\leq i\leq m$. Moreover, for $m+1\leq i\leq r-2$, $$g(D_i)=\left(\Delta_{\p^{i-1}}^{-|\p|}\Delta_{\p^{i}}^{|\p|^2+|\p|+1}\Delta_{\p^{i+1}}^{-|\p|^2-|\p|-1}\Delta_{\p^{i+2}}^{|\p|}\right)^{\frac{1}{(q-1)(|\p|^2-1)|\p|^{i}}}.$$ So, $\ord(\overline{D_i})\mid |\p|^{i}M(\p)$ for $m+1 \leq i\leq r-2$. To find lower bounds, define $e_k := e(k+1,\pi^k) \in E(\Tree)$ for $k \geq 1$ by section \ref{Section: Evaluation of harmonic cochains}, then consider the matrix $[\widetilde{r}(g(D_i))(e_{j\delta})]_{1\leq i, j\leq r-2}$. By lemma \ref{evaluations} and \ref{lemma of the structure of C(p^r)}, it takes the form:
$$\begin{bmatrix}
\frac{|\p|^{r-1} - |\p|^{r-2} + 1}{|\p|^{r-1}} & \ast & \ast & \ast & \ast & \ast & \ast\\ \cline{1-1}
|\p|-1 & \multicolumn{1}{|c}{\frac{|\p|^{r-2} - |\p|^{r-3} + 1}{|\p|^{r-2}}} & \ast & \ast & \ast & \ast & \ast\\ \cline{2-2}
\vdots & \ddots & \multicolumn{1}{|c}{\ddots} & \ast & \ast & \ast & \ast\\ \cline{3-3}
(|\p|-1)|\p|^{m-2} & \cdots & |\p|-1 & \multicolumn{1}{|c}{\frac{|\p|^{r-m} - |\p|^{r-m-1} + 1}{|\p|^{r-m}}} & \ast & \ast & \ast\\ \hline
 &  &  & \bord & \frac{1}{|\p|^{m+1}} & \ast & \ast\\ \cline{5-5}
 &  & \bigzero &  & \bord & \ddots & \ast\\ \cline{6-6}
 &  &  &  &  & \bord & \frac{1}{|\p|^{r-2}}\\ \cline{7-7}
\end{bmatrix},$$ where the $(i,j)$-entry is
$$\begin{cases}
(|\p|^{r-i} - |\p|^{r-i-1} + 1)/|\p|^{r-i},  & \text{if $1\leq i=j \leq m$.} \\
1/|\p|^{i},  & \text{if $m+1\leq i=j \leq r-2$.} \\
(|\p|-1)|\p|^{i-j-1},  & \text{if $2\leq i \leq m$ and $j<i$.}
\end{cases}$$
If $\delta \geq 2$, consider an additional matrix $[\widetilde{r}(g(D_i))(e_{(j-1)\delta+1})]_{1\leq i, j\leq r-2}$, which takes the following form by lemma \ref{evaluations} and \ref{lemma of the structure of C(p^r)}:
\begin{small}
$$\begin{bmatrix}
\frac{(|\p|^{r-1} - |\p|^{r-2})M(\p) + 1}{q|\p|^{r-2}M(\p)} & \ast & \ast & \ast & \ast & \ast & \ast\\ \cline{1-1}
(|\p|-1)\frac{|\p|}{q} & \multicolumn{1}{|c}{\ddots} & \ast & \ast & \ast & \ast & \ast\\ \cline{2-2}
\vdots & \ddots & \multicolumn{1}{|c}{\ddots} & \ast & \ast & \ast & \ast\\ \cline{3-3}
(|\p|-1)\frac{|\p|^{m-1}}{q} & \cdots & (|\p|-1)\frac{|\p|}{q} & \multicolumn{1}{|c}{\frac{(|\p|^{r-m} - |\p|^{r-m-1})M(\p) + 1}{q|\p|^{r-m-1}M(\p)}} & \ast & \ast & \ast\\ \hline
 &  &  & \bord & \frac{1}{q|\p|^{m}M(\p)} & \ast & \ast\\ \cline{5-5}
 &  & \bigzero &  & \bord & \ddots & \ast\\ \cline{6-6}
 &  &  &  &  & \bord & \frac{1}{q|\p|^{r-3}M(\p)}\\ \cline{7-7}
\end{bmatrix},$$
\end{small}
where the $(i,j)$-entry is
$$\begin{cases}
((|\p|^{r-i} - |\p|^{r-i-1})M(\p) + 1)/(q|\p|^{r-i-1}M(\p)),  & \text{if $1\leq i=j \leq m$.} \\
1/(q|\p|^{i-1}M(\p)),  & \text{if $m+1\leq i=j \leq r-2$.} \\
(|\p|-1)|\p|^{i-j}/q,  & \text{if $2\leq i \leq m$ and $j<i$.}
\end{cases}$$
From the denominators of the diagonal entries of the above matrices, we have
\begin{enumerate}
    \item For $1\leq i\leq m$, the upper bound $|\p|^{r-i}M(\p)$ of $\ord(\overline{D_i})$ is optimal.
    \item For $m+1\leq i\leq r-2$, the upper bound $|\p|^{i}M(\p)$ of $\ord(\overline{D_i})$ is optimal.
\end{enumerate}
Moreover, all the entries below the diagonal in both matrices are integers. Then by lemma \ref{linear independence lemma}, $$\langle\overline{D_i}:~1\leq i\leq r-2\rangle=\bigoplus_{1\leq i\leq r-2}\langle \overline{D_i}\rangle.$$ If $\delta = 1$, then $\ord(\overline{D_{r-1}})=\ord(\overline{D_0})=1$ by lemma \ref{sublemma 1} and \ref{sublemma 2}, which completes the proof. Assume that $\delta \geq 2$. For $1 \leq j \leq r-2$, we have
\begin{enumerate}
    \item $\widetilde{r}(g(D_{r-1}))(e_{(j-1)\delta+1})\in \mathbb{Z}$ (see Appendix \ref{Appendix 1}).
    \item $\widetilde{r}(g(D_0))(e_{(j-1)\delta+1})
    =\frac{1}{|\p|-1}\cdot \widetilde{r}\left(\frac{\Delta_{\p^{r-1}}}{\Delta_{\p^r}}\right)(e_{(j-1)\delta+1})
    =(q-1)\frac{|\p|^{r-j}}{q}\in \mathbb{Z}.$
\end{enumerate}
Note that $p$ does not divide the orders of $\overline{D_{r-1}}$ and $\overline{D_0}$ by lemma \ref{sublemma 1} and \ref{sublemma 2}. By the second matrix above and lemma \ref{linear independence lemma}, this implies that
$$\left(\bigoplus_{1\leq i\leq r-2}\langle \overline{D_i}\rangle\right)\cap \langle \overline{D_{r-1}}, \overline{D_0}\rangle = \{0\}.$$
Now, observe that
\begin{enumerate}
    \item The denominator of $\widetilde{r}(g(D_{r-1}))(e_{(r-2)\delta+1})$ is $M(\p)$ (see Appendix \ref{Appendix 1}), which is equal to $\ord(\overline{D_{r-1}})$ by lemma \ref{sublemma 1}.
    \item $\widetilde{r}(g(D_{0}))(e_{(r-2)\delta+1}) = (q-1)\frac{|\p|}{q}\in \mathbb{Z}$.
\end{enumerate}
By lemma \ref{linear independence lemma}, $\langle \overline{D_{r-1}}\rangle \cap \langle \overline{D_0}\rangle = \{0\}$. It remains to find the order of $\overline{D_0}$, which is done by lemma \ref{sublemma 2}. In conclusion, the group $\CC(\p^r)$ can be expressed as a direct sum of cyclic groups generated by $\overline{D_i}$'s, which completes the proof.
\end{proof}

\begin{Rem}
A basis $\mathcal{B}$ of $\CC(\p^r)$ is established in the above proof by the following strategy: write $\CC(\p^r)=\langle\overline{C_1},\cdots,\overline{C_{r-1}},\overline{C_0}\rangle = \langle\overline{D_1},\cdots,\overline{D_{r-1}},\overline{D_0}\rangle =: \langle\mathcal{B}\rangle$. The modified generators $\overline{D_i}$ in the order of $i=1,\cdots, r-1, 0$ are constructed one after one using lemma \ref{linear independence lemma} so that each modified generator has no non-trivial relation with all its preceding $\overline{D_i}$'s in $\CC(\p^r)$.
\end{Rem}

\subsection{Proof of lemmas} \label{Proof of Lemmas}
Fix a prime $\p\in A$ and $r\geq 2$. We prove the following lemmas for theorem \ref{Main Theorem}.
\begin{lem} \label{sublemma 1}
The order of $\overline{D_{r-1}}$ in $\CC(\p^r)$ divides $M(\p)$.
\end{lem}

\begin{proof}
In the following, we compute $\displaystyle g(D_{r-1}) = \left(\prod_{\text{monic }\dd|\p^r}\Delta_\dd^{r_\dd}\right)^{\frac{1}{(q-1)(|\p|^2-1)}}$.
\begin{enumerate}
\item If $r = 2$, then $g(D_{1}) = \left(\Delta_{}^{-1}\Delta_{\p}^{|\p|+1}\Delta_{\p^2}^{-|\p|}\right)^{\frac{1}{(q-1)(|\p|^2-1)}}$.
\item If $r = 3$, then $g(D_{2}) = \left(\Delta_{}^{|\p|^2-1}\Delta_{\p}^{-|\p|^3}\Delta_{\p^2}^{|\p|+1}\Delta_{\p^3}^{|\p|^3-|\p|^2-|\p|}\right)^{\frac{1}{(q-1)(|\p|^2-1)}}$.
\item If $r = 4$, then
$$r_\dd = \begin{cases}
|\p|^2-1,  & \text{if $\dd=1$.} \\
-|\p|^3-|\p|^2+|\p|+1,  & \text{if $\dd=\p^{}$.} \\
|\p|^3+|\p|^2-|\p|-2,  & \text{if $\dd=\p^{2}$.} \\
-|\p|^3-|\p|^2+2|\p|+2,  & \text{if $\dd=\p^{3}$.} \\
|\p|^3-2|\p|,  & \text{if $\dd=\p^{4}$.}
\end{cases}$$
\item If $r = 5$, then
$$r_\dd = \begin{cases}
|\p|^2-1,  & \text{if $\dd=1$.} \\
-|\p|^3-|\p|^2+|\p|+1,  & \text{if $\dd=\p^{}$.} \\
|\p|^3+2|\p|^2-|\p|-2,  & \text{if $\dd=\p^{2}$.} \\
-2|\p|^3-2|\p|^2+|\p|+2,  & \text{if $\dd=\p^{3}$.} \\
2|\p|^3+|\p|^2-|\p|,  & \text{if $\dd=\p^{4}$.} \\
-|\p|^2,  & \text{if $\dd=\p^{5}$.}
\end{cases}$$
\item If $r = 6$, then
$$r_\dd = \begin{cases}
|\p|^2-1,  & \text{if $\dd=1$.} \\
-|\p|^3-|\p|^2+|\p|+1,  & \text{if $\dd=\p^{}$.} \\
|\p|^3+|\p|^2-|\p|-1,  & \text{if $\dd=\p^{2}$.} \\
-|\p|^3+|\p|,  & \text{if $\dd=\p^{3}$.} \\
-|\p|^2,  & \text{if $\dd=\p^{4}$.} \\
|\p|^3+1,  & \text{if $\dd=\p^{5}$.} \\
-|\p|,  & \text{if $\dd=\p^{6}$.}
\end{cases}$$
\item If $r\geq 7$ and $r\equiv 3 \mod 4$, then
$$r_\dd = \begin{cases}
|\p|^2-1,  & \text{if $\dd=1$ or $\p^{2}$.} \\
-|\p|^3-|\p|^2+|\p|+1,  & \text{if $\dd=\p^{}$.} \\
|\p|^{\frac{r+1}{2}}-|\p|^{\frac{r-1}{2}}+|\p|-1,  & \text{if $\dd=\p^{\frac{r-1}{2}}$.} \\
-|\p|^{\frac{r-1}{2}}+|\p|^{\frac{r-3}{2}}-|\p|^2+|\p|,  & \text{if $\dd=\p^{\frac{r+1}{2}}$.} \\
-|\p|,  & \text{if $\dd=\p^{r-2}$.} \\
|\p|^{\frac{r-1}{2}}-|\p|^{\frac{r-3}{2}}+2,  & \text{if $\dd=\p^{r-1}$.} \\
-|\p|^{\frac{r+1}{2}}+|\p|^{\frac{r-1}{2}}+|\p|^3-2|\p|,  & \text{if $\dd=\p^{r}$.} \\
0, & \text{otherwise.}
\end{cases}$$
\item If $r\geq 8$ and $r\equiv 0 \mod 4$, then
$$r_\dd = \begin{cases}
|\p|^2-1,  & \text{if $\dd=1$ or $\p^{2}$.} \\
-|\p|^3-|\p|^2+|\p|+1,  & \text{if $\dd=\p^{}$.} \\
|\p|^{\frac{r}{2}}-|\p|^{\frac{r}{2}-1}-|\p|^2+|\p|,  & \text{if $\dd=\p^{\frac{r}{2}-1}$.} \\
-|\p|^{\frac{r}{2}-1}+|\p|^{\frac{r}{2}-2}+|\p|^3+|\p|^2-2,  & \text{if $\dd=\p^{\frac{r}{2}}$.} \\
(-1)^k\cdot 2(|\p|^3+|\p|^2-|\p|-1),  & \text{if $\dd=\p^k$ and} \\
& \frac{r}{2}+1\leq k \leq r-3. \\
2|\p|^3+|\p|^2-2|\p|-2,  & \text{if $\dd=\p^{r-2}$.} \\
|\p|^{\frac{r}{2}-1}-|\p|^{\frac{r}{2}-2}-|\p|^3-|\p|^2+|\p|+3,  & \text{if $\dd=\p^{r-1}$.} \\
-|\p|^{\frac{r}{2}}+|\p|^{\frac{r}{2}-1}+|\p|^3+|\p|^2-3|\p|,  & \text{if $\dd=\p^{r}$.} \\
0, & \text{otherwise.}
\end{cases}$$
\item If $r\geq 9$ and $r\equiv 1 \mod 4$, then
$$r_\dd = \begin{cases}
|\p|^2-1,  & \text{if $\dd=1$ or $\p^{2}$.} \\
-|\p|^3-|\p|^2+|\p|+1,  & \text{if $\dd=\p^{}$.} \\
|\p|^{\frac{r+1}{2}}-|\p|^{\frac{r-1}{2}}+2|\p|^2-|\p|-1,  & \text{if $\dd=\p^{\frac{r-1}{2}}$.} \\
-|\p|^{\frac{r-1}{2}}+|\p|^{\frac{r-3}{2}}-2|\p|^3-|\p|^2+|\p|+2,  & \text{if $\dd=\p^{\frac{r+1}{2}}$.} \\
(-1)^k\cdot 2(|\p|^3+|\p|^2-|\p|-1),  & \text{if $\dd=\p^k$ and} \\
& \frac{r+3}{2}\leq k \leq r-3. \\
-2|\p|^3-2|\p|^2+|\p|+2,  & \text{if $\dd=\p^{r-2}$.} \\
|\p|^{\frac{r-1}{2}}-|\p|^{\frac{r-3}{2}}+2|\p|^3,  & \text{if $\dd=\p^{r-1}$.} \\
-|\p|^{\frac{r+1}{2}}+|\p|^{\frac{r-1}{2}}+|\p|^3-2|\p|^2,  & \text{if $\dd=\p^{r}$.} \\
0, & \text{otherwise.}
\end{cases}$$
\item If $r\geq 10$ and $r\equiv 2 \mod 4$, then
$$r_\dd = \begin{cases}
|\p|^2-1,  & \text{if $\dd=1$ or $\p^{2}$.} \\
-|\p|^3-|\p|^2+|\p|+1,  & \text{if $\dd=\p^{}$.} \\
|\p|^{\frac{r}{2}}-|\p|^{\frac{r}{2}-1}+|\p|^2-|\p|,  & \text{if $\dd=\p^{\frac{r}{2}-1}$.} \\
-|\p|^{\frac{r}{2}-1}+|\p|^{\frac{r}{2}-2}-|\p|^3+|\p|^2,  & \text{if $\dd=\p^{\frac{r}{2}}$.} \\
-|\p|^2,  & \text{if $\dd=\p^{r-2}$.} \\
|\p|^{\frac{r}{2}-1}-|\p|^{\frac{r}{2}-2}+|\p|^3-|\p|^2+|\p|+1,  & \text{if $\dd=\p^{r-1}$.} \\
-|\p|^{\frac{r}{2}}+|\p|^{\frac{r}{2}-1}+|\p|^3-|\p|^2-|\p|,  & \text{if $\dd=\p^{r}$.} \\
0, & \text{otherwise.}
\end{cases}$$
\end{enumerate}
Since $r_\dd\in \mathbb{Z}$ for all monic $\dd\mid\p^r$, we have $\ord(\overline{D_{r-1}})\mid (q-1)(|\p|^2-1)$, which implies that $p\nmid \ord(\overline{D_{r-1}})$. Now, the proof is complete by lemma \ref{lemma of the structure of C(p^r)}.
\end{proof}

\begin{lem} \label{sublemma 2}
The order of $\overline{D_0}$ in $\CC(\p^r)$ is $N(\p)$.
\end{lem}

\begin{proof}
We have $$g(D_0) = \left(\frac{\Delta_{\p^{r-1}}}{\Delta_{\p^r}}\right)^{\frac{1}{|\p|-1}}.$$
Recall that $\frac{\Delta_{\p^{r-1}}}{\Delta_{\p^r}}$ has a maximal $k$-th root $f(z) := D_\p(\p^{r-1}z)$ (up to constant multiple) in $\mathcal{O}(\Omega)^\ast$, where $$k=\begin{cases}
(q-1)(q^2-1),  & \text{if $\deg(\p)$ is even.} \\
(q-1)^2,  & \text{otherwise.}
\end{cases}$$ By lemma \ref{transformation law}, $f(\gamma z) = \omega_{\p}(\gamma) f(z)$ for $\gamma\in \Gamma_0(\p^r)$. Since $f^{q-1}$ is the minimal power of $f$ which is invariant under $\Gamma_0(\p^r)$, and
$$g(D_0) = \text{const.}~f^{\frac{k}{|\p|-1}}=\text{const.}~(f^{q-1})^{\frac{k/(q-1)}{|\p|-1}},$$
we see that $\frac{|\p|-1}{k/(q-1)}$ is the order of $\overline{D_0}$ in $\CC(\p^r)$.
\end{proof}

\begin{appendices}
\section{Computational results on $\widetilde{r}(g(D_{r-1}))$} \label{Appendix 1}
Fix a prime $\p\in A$ with $\delta:=\deg(\p)\geq 2$ and $r\geq 2$. Recall that the group $\CC(\p^r)$ is generated by $\{\overline{C_0}, \overline{C_1}, \cdots, \overline{C_{r-1}}\}$. Define $e_k := e(k+1,\pi^k) \in E(\Tree)$ for $k \geq 1$. Using lemma \ref{evaluations} and \ref{lemma of the structure of C(p^r)}, for $1\leq j \leq r-1$, we have
\begin{enumerate}
    \item If $1\leq i \leq r-1$ with $m(i):=\min\{i,r-i\}$, then $$\widetilde{r}(g(C_i))(e_{(j-1)\delta+1})
    =\begin{cases}
    \frac{(|\p|^{r-i}-|\p|^{r-i-1})(|\p|^2-1)+q^2-1}{q|\p|^{r-m(i)-1}(|\p|^{2}-1)},  & \text{if $j = i$.} \\
    \frac{(|\p|^{r-i}-|\p|^{r-i-1})(|\p|^2-1)+|\p|^2-q^2}{q|\p|^{r-m(i)}(|\p|^{2}-1)},  & \text{if $j = i+1$.} \\
    \frac{(|\p|-1)|\p|^{m(i)-j}}{q},  & \text{otherwise.}
    \end{cases}$$
    \item If $\lfloor \frac{r+1}{2} \rfloor \leq i \leq r-2$, then $$\widetilde{r}(g(C_{i}-|\p|C_{i+1}))(e_{(j-1)\delta+1})
    =\begin{cases}
    \frac{q^2-1}{q|\p|^{i-1}(|\p|^2-1)},  & \text{if $j=i$.} \\
    \frac{|\p|-q^2}{q|\p|^{i}(|\p|-1)},  & \text{if $j=i+1$.} \\
    -\frac{|\p|^2-q^2}{q|\p|^{i}(|\p|^2-1)},  & \text{if $j=i+2$.} \\
    0,  & \text{otherwise.}
    \end{cases}$$
\end{enumerate}
From the above and the definition of $D_{r-1}$ in section \ref{section: Main Theorem}, we compute that
\begin{enumerate}
\item If $r = 2$, then $q\cdot \widetilde{r}(g(D_{r-1}))(e_{1})=|\p|-1+\frac{1}{M(\p)}$.
\item If $r\geq 3$ and $r\equiv 3 \mod 4$, then $q\cdot \widetilde{r}(g(D_{r-1}))(e_{(j-1)\delta+1})$ is 
\begin{small}
\begin{align*}
    \begin{cases}
    \!\begin{aligned}
    & (|\p|^{\frac{3r-5}{2}}-|\p|^{r}-|\p|^{r-1}+|\p|^{\frac{r-1}{2}})(|\p|-1)+|\p|-q^2, \end{aligned} & \text{if $j =1$.} \\
    \!\begin{aligned}
    &(|\p|^{\frac{3r-7}{2}}-|\p|^{r-1}-|\p|^{r-2}+|\p|^{\frac{r-3}{2}})(|\p|-1)-|\p|+q^2, \end{aligned} & \text{if $j =2$ and $r \geq 7$.} \\
    \!\begin{aligned}
    &(|\p|^{\frac{3r-3}{2}-j}-|\p|^{r-j+1}-|\p|^{r-j}+|\p|^{\frac{r+1}{2}-j}+q^2|\p|^{j-3})(|\p|-1),\end{aligned} & \text{if $3 \leq j \leq \frac{r-1}{2}$.} \\
    \!\begin{aligned}
    & (|\p|^{r-2}-|\p|^{\frac{r+1}{2}}-|\p|^{\frac{r-1}{2}})(|\p|-1)+\sum_{0\leq i\leq \frac{r-7}{2}}(|\p|^2-q^2)(-|\p|)^i,
    \end{aligned} & \text{if $j = \frac{r+1}{2}$ and $r \geq 7$.} \\
    \!\begin{aligned}
    (|\p|^{\frac{3r-3}{2}-j}-|\p|^{r-j+1}-|\p|^{r-j})(|\p|-1),
    \end{aligned}  & \text{if $\frac{r+3}{2} \leq j \leq r-2$.} \\
    \!\begin{aligned}
    |\p|^{\frac{r+1}{2}}-|\p|^{\frac{r-1}{2}}-|\p|^3+|\p|+\frac{|\p|}{M(\p)},
    \end{aligned}  & \text{if $j = r-1$.}
    \end{cases}
\end{align*}
\end{small}
\item If $r\geq 4$ and $r\equiv 0 \mod 4$, then $q\cdot \widetilde{r}(g(D_{r-1}))(e_{(j-1)\delta+1})$ is 
\begin{small}
\begin{align*}
    \begin{cases}
    \!\begin{aligned}
    &(|\p|^{\frac{3r}{2}-3}-|\p|^{r}-|\p|^{r-1}+|\p|^{\frac{r}{2}})(|\p|-1)+|\p|-q^2, \end{aligned} & \text{if $j =1$.} \\
    \!\begin{aligned}
    &(|\p|^{\frac{3r}{2}-4}-|\p|^{r-1}-|\p|^{r-2}+|\p|^{\frac{r}{2}-1})(|\p|-1)-|\p|+q^2, \end{aligned} & \text{if $j =2$.} \\
    \!\begin{aligned}
    &(|\p|^{\frac{3r}{2}-j-2}-|\p|^{r-j+1}-|\p|^{r-j}+|\p|^{\frac{r}{2}-j+1}+q^2|\p|^{j-3})(|\p|-1),\end{aligned} & \text{if $3 \leq j \leq \frac{r}{2}-1$.} \\
    \!\begin{aligned}
    &(|\p|^{r-2}-|\p|^{\frac{r}{2}+1}-|\p|^{\frac{r}{2}}+|\p|)(|\p|-1)-|\p|+q^2\\
    & +\sum_{0\leq i\leq \frac{r-8}{2}}(|\p|^2-q^2)(-|\p|)^i,
    \end{aligned} & \text{if $j = \frac{r}{2}$.} \\
    \!\begin{aligned}
    &(|\p|^{\frac{3r}{2}-j-2}-|\p|^{r-j+1}-|\p|^{r-j})(|\p|-1)+(-1)^{j+1}2(|\p|-q^2),
    \end{aligned} & \text{if $\frac{r}{2}+1 \leq j \leq r-2$.} \\
    \!\begin{aligned}
    &|\p|^{\frac{r}{2}}-|\p|^{\frac{r}{2}-1}-|\p|^3+3|\p|-2q^2+\frac{|\p|^2}{M(\p)},
    \end{aligned} & \text{if $j = r-1$.}
    \end{cases}
\end{align*}
\end{small}
\item If $r\geq 5$ and $r\equiv 1 \mod 4$, then $q\cdot \widetilde{r}(g(D_{r-1}))(e_{(j-1)\delta+1})$ is
\begin{small}
\begin{align*}
    \begin{cases}
    \!\begin{aligned}
    & (|\p|^{\frac{3r-5}{2}}-|\p|^{r}-|\p|^{r-1}+|\p|^{\frac{r-1}{2}})(|\p|-1)+|\p|-q^2, \end{aligned} & \text{if $j =1$.} \\
    \!\begin{aligned}
    &(|\p|^{\frac{3r-7}{2}}-|\p|^{r-1}-|\p|^{r-2}+|\p|^{\frac{r-3}{2}})(|\p|-1)-|\p|+q^2, \end{aligned} & \text{if $j =2$.} \\
    \!\begin{aligned}
    &(|\p|^{\frac{3r-3}{2}-j}-|\p|^{r-j+1}-|\p|^{r-j}+|\p|^{\frac{r+1}{2}-j}+q^2|\p|^{j-3})(|\p|-1),\end{aligned} & \text{if $3 \leq j \leq \frac{r-1}{2}$.} \\
    \!\begin{aligned}
    & (|\p|^{r-2}-|\p|^{\frac{r+1}{2}}-|\p|^{\frac{r-1}{2}})(|\p|-1)+2(|\p|-q^2)\\
    & -\sum_{0\leq i\leq \frac{r-7}{2}}(|\p|^2-q^2)(-|\p|)^i,
    \end{aligned} & \text{if $j = \frac{r+1}{2}$.} \\
    \!\begin{aligned}
    & (|\p|^{\frac{3r-3}{2}-j}-|\p|^{r-j+1}-|\p|^{r-j})(|\p|-1)+(-1)^{j+1}2(|\p|-q^2),
    \end{aligned}& \text{if $\frac{r+3}{2} \leq j \leq r-2$.} \\
    \!\begin{aligned}
    & |\p|^{\frac{r+1}{2}}-|\p|^{\frac{r-1}{2}}-|\p|^3-|\p|+2q^2+\frac{|\p|}{M(\p)},
    \end{aligned}& \text{if $j = r-1$.}
    \end{cases}
\end{align*}
\end{small}
\item If $r\geq 6$ and $r\equiv 2 \mod 4$, then $q\cdot \widetilde{r}(g(D_{r-1}))(e_{(j-1)\delta+1})$ is 
\begin{small}
\begin{align*}
    \begin{cases}
    \!\begin{aligned}
    & (|\p|^{\frac{3r}{2}-3}-|\p|^{r}-|\p|^{r-1}+|\p|^{\frac{r}{2}})(|\p|-1)+|\p|-q^2, \end{aligned} & \text{if $j =1$.} \\
    \!\begin{aligned}
    &(|\p|^{\frac{3r}{2}-4}-|\p|^{r-1}-|\p|^{r-2}+|\p|^{\frac{r}{2}-1})(|\p|-1)-|\p|+q^2, \end{aligned} & \text{if $j =2$.} \\
    (|\p|^{\frac{3r}{2}-j-2}-|\p|^{r-j+1}-|\p|^{r-j}+|\p|^{\frac{r}{2}-j+1}+q^2|\p|^{j-3})(|\p|-1), & \text{if $3 \leq j \leq \frac{r}{2}-1$.}\\
    \!\begin{aligned}
    & (|\p|^{r-2}-|\p|^{\frac{r}{2}+1}-|\p|^{\frac{r}{2}})(|\p|-1)+|\p|^2-q^2\\
    & -\sum_{0\leq i\leq \frac{r-8}{2}}(|\p|^2-q^2)(-|\p|)^i,
    \end{aligned} & \text{if $j = \frac{r}{2}$.} \\
    \!\begin{aligned}
    (|\p|^{\frac{3r}{2}-j-2}-|\p|^{r-j+1}-|\p|^{r-j})(|\p|-1),
    \end{aligned}& \text{if $\frac{r}{2}+1 \leq j \leq r-2$.} \\
    \!\begin{aligned}
    |\p|^{\frac{r}{2}}-|\p|^{\frac{r}{2}-1}-|\p|^3+|\p|+\frac{|\p|^2}{M(\p)},
    \end{aligned}& \text{if $j = r-1$.}
    \end{cases}
\end{align*}
\end{small}
\end{enumerate}
\end{appendices}

\bibliographystyle{acm}
\bibliography{Bibliography}

\end{document}